\documentclass[a4paper,11pt]{amsart}
\usepackage{amssymb,amsmath,amsthm}
\usepackage{graphicx,color}
\usepackage{stackrel}
\usepackage[pdftex]{hyperref}

\newcommand{\ie}{\emph{i.e.}}

\newcommand{\cf}{\emph{cf}}

\newcommand{\Real}{\mathbb{R}}
\newcommand{\Com}{\mathbb{C}}
\newcommand{\sii}{L^2}
\newcommand{\Dom}{\mathrm{Dom}}

\newcommand{\eps}{\varepsilon}
\newcommand{\Cut}{\mathcal{C}}

\newcommand{\dist}{\mathop{\mathrm{dist}}\nolimits}

\newcommand{\ds}{\displaystyle}

\newtheorem{Theorem}{Theorem}

\newtheorem{Proposition}{Proposition}
\newtheorem{Conjecture}{Conjecture}
\theoremstyle{definition}
\newtheorem{Remark}{Remark}

\usepackage[normalem]{ulem}
\definecolor{DarkGreen}{rgb}{0,0.5,0.1} 

\newcommand\soutD{\bgroup\markoverwith
{\textcolor{DarkGreen}{\rule[.5ex]{2pt}{1pt}}}\ULon}
\newcommand{\Hm}[1]{\leavevmode{\marginpar{\tiny%
$\hbox to 0mm{\hspace*{-0.5mm}$\leftarrow$\hss}%
\vcenter{\vrule depth 0.1mm height 0.1mm width \the\marginparwidth}%
\hbox to
0mm{\hss$\rightarrow$\hspace*{-0.5mm}}$\\\relax\raggedright #1}}}

\begin{document}
%
\title{The first Robin eigenvalue with negative boundary parameter}
\author{Pedro Freitas \ and \ David Krej\v{c}i\v{r}\'{\i}k}

\address{Department of Mathematics, Faculty of Human Kinetics \&
 Group of Mathematical Physics, 
 Universidade de Lisboa, Complexo Interdisciplinar, 
Av.~Prof.\ Gama Pinto~2, 
 P-1649-003 Lisboa, Portugal}

\email{freitas@cii.fc.ul.pt}

\address{
Department of Theoretical Physics,
Nuclear Physics Institute,
Academy of Sciences,
250\,68 \v{R}e\v{z}, Czech Republic
}
\email{krejcirik@ujf.cas.cz}

\date{15 May 2014}

\thanks{The research of the first author was partially supported 
by FCT's project PEst-OE/MAT/UI0208/2011.
The research of the second author was supported by 
RVO61389005 and the GACR grant No.\ 14-06818S}

\begin{abstract}
We give a counterexample to the long standing conjecture 
that the ball maximises the first eigenvalue of the Robin
eigenvalue problem with negative parameter among domains of the same volume. Furthermore, we show that the
conjecture holds in two dimensions provided that the boundary parameter is small. 
This is the first known example
within the class of isoperimetric spectral problems for the first eigenvalue of the Laplacian where the ball
is not an optimiser.
\end{abstract}

\maketitle
\section{Introduction}
%
Lord Rayleigh's book \emph{The theory of sound}~\cite{Rayleigh_1877} was the starting point for the study
of what are now called isoperimetric spectral inequalities, the most famous of which being the Rayleigh-Faber-Krahn
inequality. This states that the first Dirichlet eigenvalue of a domain $\Omega$ is always greater than or equal
to the first Dirichlet eigenvalue for a ball with the same volume. The first proof of this result was given by
Faber~\cite{Faber_1923} and Krahn~\cite{Krahn_1924} independently, nearly fifty years after Rayleigh's
conjectured it based on a few explicitly computed examples.

More generally, such inequalities relate the first (non-trivial) eigenvalue of the Laplace and other elliptic operators to the
volume of the domain and to the eigenvalue of one particular domain for which equality is attained. Since the appearance of Faber
and Krahn's work, these inequalities have been extended to the case of Neumann boundary conditions (Szeg\"{o}~\cite{Szego_1954}
and Weinberger~\cite{Weinberger_1956}, for two and higher dimensions, respectively) and Robin boundary conditions with
a positive boundary parameter (Bossel~\cite{Bossel_1986} and Daners~\cite{Daners_2006}, again for two and higher dimensions,
respectively).

A common feature of all these inequalities and one which is also shared by those for some higher order operators such as
the bi-Laplacian (at least in dimensions two and three 
-- see~\cite{Nadirashvili_1995} and~\cite{Ashbaugh-Benguria_1995}, respectively), 
is the fact that the optimal domain turns out to be the ball in all cases, 
thus mimicking the classical geometric isoperimetric inequality
between the surface area of a domain and its volume. It would thus be natural to expect that, in the only situation for the
Laplacian for which a result is still lacking, namely, the Robin problem with a negative boundary parameter, the optimal domain
should also be the ball. Indeed, this was conjectured by Bareket~\cite{Bareket_1977} as early as 1977, who
proved that the result holds within a class of \emph{nearly circular domains} and for a certain range of the boundary parameter.

To make matters more precise, given a domain $\Omega$ in $\Real^d$, 
we are interested in the problem
\begin{equation}
    \label{eq:robin}
    	\left\{
    \begin{aligned}
        -\Delta u &= \lambda u &\quad &\text{in $\Omega$} \,,\\
        \frac{\partial u}{\partial n}+\alpha u&=0 & &
        \text{on $\partial\Omega$} \,,
    \end{aligned}
\right.
\end{equation}
where~$n$ is the outer unit normal to $\Omega$ 
and the boundary parameter $\alpha$ is a real constant. While in the case
of positive $\alpha$ 
we have that the ball minimises the first eigenvalue 
$\lambda_{1}^\alpha(\Omega)$ of \eqref{eq:robin}
in a similar way to the classical Rayleigh-Faber-Krahn inequality,
here due to the fact that $\alpha$ is negative we now want to maximise this eigenvalue. 
More precisely, given a positive number~$\omega_0$,
we are interested in the problem
\begin{equation}
 \sup_{|\Omega|=\omega_{0}}\lambda_{1}^\alpha(\Omega)
  \,,
\end{equation}
where it should be emphasised that, at least \emph{a priori}, 
this also depends on the parameter $\alpha$ -- note that while the optimisers of
the first two Robin eigenvalues for positive $\alpha$ do not depend on this parameter, there is numerical evidence that this is not the
case for higher eigenvalues~\cite{Antunes-Freitas-Kennedy_2013}.
Bareket's conjecture may then be stated as follows.
\begin{Conjecture}\label{Conj}
Let $\alpha \leq 0$.
For any bounded smooth domain~$\Omega$,
\begin{equation}\label{Eq.Conj}
  \lambda_1^\alpha(\Omega) \leq \lambda_1^\alpha(B)
  \,,
\end{equation}
where~$B$ is a ball of the same volume as~$\Omega$.
\end{Conjecture}

In the intervening years since Bareket's paper not much progress has been made, although the conjecture has been
revived recently in both~\cite{Brock-Daners_2007} and~\cite{Daners_2013}, 
while Ferone, Nitsch and Trombetti have shown that it holds
within the ``class of Lipschitz sets which are `close'
to a ball in a Hausdorff metric sense''~\cite{Ferone-Nitsch-Trombetti}.

The purpose of this paper is twofold. On the one hand, we provide a counterexample showing that the conjecture does not hold
for general domains. This is done by determining a two-term asymptotic expansion in $\alpha$ for both balls and spherical shells,
yielding that the latter must provide a larger value than the former for sufficiently large (negative) $\alpha$ and may
be stated as follows.
\begin{Theorem}[General domains]\label{mainthm1}
Conjecture~\ref{Conj} does not hold among the class of bounded domains in $\Real^d$ ($d\geq2$); more precisely,~\eqref{Eq.Conj}~is
violated whenever~$\Omega$ is a spherical shell of the same volume as the ball~$B$
and for sufficiently large negative~$\alpha$.
\end{Theorem}

On the other hand, we show that in the two-dimensional case the ball is still the maximiser
provided that $\alpha$ is negative and sufficiently small in absolute value.

\begin{Theorem}[Planar domains]\label{mainthm2}
For bounded planar domains of class~$C^2$ and fixed area, 
there exists a negative number~$\alpha_{*}$, 
depending only on the area, such that~\eqref{Eq.Conj} holds for all
$\alpha \in [\alpha_{*},0]$.
\end{Theorem}

The fact that the ball stops being the maximiser
for a critical value of~$\alpha$ is quite surprising 
in view of the isoperimetric spectral properties of the similar
problems described above and it is, to the best of our knowledge, 
the first problem of its type for which this happens. 
Based on Theorems~\ref{mainthm1} and~\ref{mainthm2} we conjecture, 
however, that the global maximiser is still symmetric 
for reflections with respect to all hyperplanes, 
being a ball for small (absolute) values of~$\alpha$ 
and bifurcating to annuli at a critical value of $\alpha$
(depending on the volume).

The structure of the paper is as follows. 
In the next section we lay out the problem setting. 
Although there have been many papers dealing with the asymptotic 
behaviour of the eigenvalues of problem~\eqref{eq:robin} as the
parameter~$\alpha$ becomes large and negative, 
such as~\cite{Daners-Kennedy_2006,Exner-Minakov-Parnovski,
Lacey-Ockendon-Sabina_1998,Levitin-Parnovski_2008,Pankrashkin_2013}, 
none of the existing results in the literature apply directly to our problem 
in its full generality. 
For completeness, we thus derive the asymptotics necessary 
to obtain our counter-example in Section~\ref{Sec.as}.
Finally, in the last section we show that, 
for small negative~$\alpha$ and within the context of planar 
domains, the conjecture holds. 
This is done in two parts. 
By applying the method of parallel coordinates, we
first compare such domains with
annuli with Robin and Neumann boundary conditions 
in the outer and inner circles, respectively, 
thus allowing us to parametrise these domains (for this purpose) 
by their area and outer perimeter alone. 
Then we compare such annuli with the Robin
disk, proving the desired result.

\section{Preliminaries}
%
Let~$\Omega$ be a domain (\ie\ open connected set) in~$\Real^d$ with $d \geq 1$. 
For simplicity, let us assume that~$\Omega$ is bounded and Lipschitz,
so that the vector field~$n$ is defined almost everywhere in~$\partial\Omega$
and the boundary traces 
$W^{1,2}(\Omega) \hookrightarrow \sii(\partial\Omega)$ exist.
Let $\alpha \in \Real$.
We understand~\eqref{eq:robin} as a spectral problem for
the self-adjoint operator $-\Delta_\alpha^\Omega$ in $\sii(\Omega)$
associated with the closed quadratic form
\begin{equation}\label{form}
  Q_\alpha^\Omega[u] := \|\nabla u\|_{\sii(\Omega)}^2 
  +\alpha\;\!\|u\|_{\sii(\partial\Omega)}^2
  \,, \qquad 
  \Dom(Q_\alpha^\Omega) := W^{1,2}(\Omega)
  \,.
\end{equation}
If~$\Omega$ is of class~$C^2$, then 
$
  \Dom(-\Delta_\alpha^\Omega)
$
consists of functions $u \in W^{2,2}(\Omega)$
which satisfy the Robin boundary conditions 
of~\eqref{eq:robin} in the sense of traces
and the boundary value problem~\eqref{eq:robin}
can be considered in a classical setting.
Under our minimal regularity assumptions,
it still makes sense to define the lowest
point in the spectrum of $-\Delta_\alpha^\Omega$
by the variational formula
\begin{equation}\label{Rayleigh}
  \lambda_1^\alpha(\Omega)
  := \inf_{\stackrel[u\not=0]{}{u \in W^{1,2}(\Omega)}}
  \frac{Q_\alpha^\Omega[u]}
  {\,\|u\|_{\sii(\Omega)}^2}
  \,.
\end{equation}
Since the embedding $W^{1,2}(\Omega) \hookrightarrow \sii(\Omega)$ is compact,
we know that $\lambda_1^\alpha(\Omega)$ is a discrete eigenvalue
and the infimum is achieved by a function $u_1^\alpha \in W^{1,2}(\Omega)$.

Since $\lambda_1^\alpha(\Omega)$ is the lowest eigenvalue of $-\Delta_\alpha^\Omega$,
it is simple and the corresponding eigenfunction~$u_1^\alpha$
can be chosen to be positive in~$\Omega$.
We further normalise~$u_1^\alpha$ to have unit $\sii(\Omega)$ norm.
It is straightforward to verify that $\{Q_\alpha^\Omega\}_{\alpha\in\Real}$
is a holomorphic family of forms of type~(a) in the sense of Kato 
\cite[Sec.~VII.4]{Kato}. 
In fact, the boundary term in~\eqref{form} 
is relatively bounded with respect to the Neumann form~$Q_0^\Omega$,
with the relative bound equal to zero,
so that \cite[Thm.~4.8]{Kato} applies.
Consequently, $-\Delta_\alpha^\Omega$ forms 
a self-adjoint holomorphic family of operators of type~(B).
Because of the simplicity, it follows that 
$\alpha \mapsto \lambda_1^\alpha(\Omega)$ 
and $\alpha \mapsto u_1^\alpha$ are analytic functions on~$\Real$
(the latter in the topology of $W^{1,2}(\Omega)$).

Using a constant test function 
in the variational characterisation~\eqref{Rayleigh},
we get
\begin{equation}\label{bound}
  \lambda_1^\alpha(\Omega) \leq 
  \alpha \, \frac{\mathcal{H}^{d-1}(\partial\Omega)}{|\Omega|}
  \,.
\end{equation}
Here $|\cdot|$ stands for 
the $d$-dimensional Lebesgue measure
and $\mathcal{H}^{d-1}(\cdot)$
denotes the $(d-1)$-dimensional Hausdorff measure.
It follows that $\lambda_1^\alpha(\Omega)$ is negative whenever $\alpha < 0$
and 
\begin{equation}
  \lim_{\alpha \to -\infty} \lambda_1^\alpha(\Omega) = -\infty
  \,.
\end{equation}
%

\section{Asymptotics of spherical shells}\label{Sec.as}
%
In this section we assume that $\Omega \subset \Real^d$ is 
the domain enclosed by two concentric spheres of radii $r_1 < r_2$
in any dimension $d \geq 2$ ($\Omega$ would not be connected in $d=1$).
More specifically, we define 
$$
  A_{r_1,r_2} := B_{r_2} \setminus \overline{B_{r_1}}
  \,, \qquad \mbox{where} \qquad
  B_r := \{x\in\Real^d :\, |x|<r \}
$$
is an open ball. In fact, since $B_0=\varnothing$,
we may write $B_r = A_{0,r}$ 
and think of the ball as a special case of the spherical shell.
Of course, $A_{r_1,r_2}$ is an annulus if $d=2$.
We are interested in the asymptotics of $\lambda_1^\alpha(A_{r_1,r_2})$
as $\alpha \to -\infty$.
We may thus assume that~$\alpha$ is negative and set 
$$
  k:=\sqrt{-\lambda_1^\alpha(A_{r_1,r_2})} \,>\, 0
  \,.
$$ 
Recall also that the dependence of~$k$ on~$\alpha$ is smooth.

\smallskip 

\subsection*{Step 0}
We \emph{a priori} know that~$k$ tends to $+\infty$ as $\alpha \to -\infty$.
Indeed, in the present situation of spherical shells,
bound~\eqref{bound} reduces to
\begin{equation*}
  k^2 \geq -\alpha \, d \, \frac{r_2^{d-1}-r_1^{d-1}}{r_2^d-r_1^d}
\end{equation*}
and $\alpha$ is assumed to be negative.
In particular,
\begin{equation}\label{step0}
  \frac{\alpha}{k^2} = \mathcal{O}(1)
\end{equation}
as $\alpha \to -\infty$.
Henceforth, we use the big~$\mathcal{O}$ and small~$o$ notations
without further specifying the asymptotics regime $\alpha \to -\infty$.

\subsection*{From a PDE to an algebraic equation} 
Since the domain~$A_{r_1,r_2}$ is smooth,
\eqref{eq:robin}~is a classical boundary value problem. 
Because of the rotational symmetry of~$A_{r_1,r_2}$
and since $\lambda_1^\alpha(A_{r_1,r_2})$ is simple,
the corresponding eigenfunction~$u_1^\alpha$ 
is necessarily rotationally symmetric.
The eigenpair may thus be determined by solving 
\begin{equation}\label{elliptic,annulus}
\left\{
\begin{aligned}  
  -r^{-(d-1)} [r^{d-1} \psi'(r)]' &= -k^2 \psi(r) \,,
  && r \in [r_1,r_2] \,,
  \\
  -\psi'(r_1)+\alpha\psi(r_1) &= 0 \,,
  \\
  \psi'(r_2) + \alpha \;\! \psi(r_2) &=0 \,.
\end{aligned}
\right.
\end{equation}
The general solution of the differential equation in~\eqref{elliptic,annulus}
is given by
\begin{equation}\label{Bessel.torus}
  \psi(r) = r^{-\nu}
  \left[C_1 K_{\nu}(kr) + C_2 I_{\nu}(kr)\right]
  \,, \qquad
  C_1, C_2 \in \Com
  \,,
\end{equation}
where $K_\nu, I_\nu$ are modified Bessel functions 
\cite[Sec.~9.6]{Abramowitz-Stegun} and 
\begin{equation}\label{nu}
  \nu := \frac{d-2}{2}
  \,.
\end{equation}
Requiring $\psi$ to satisfy the boundary conditions 
leads us to the homogeneous algebraic system
\begin{equation}\label{homo}
\begin{pmatrix}
  m_{11} & m_{12}
  \\
  m_{21} & m_{22} 
\end{pmatrix}
\begin{pmatrix}
  C_1 
  \\  
  C_2
\end{pmatrix}
= 
\begin{pmatrix}
  0 
  \\  
  0
\end{pmatrix}
  \,, 
\end{equation}
with the entries
\begin{align*}
  m_{11} &:= -k K_{\nu}'(k r_1) + \mbox{$\frac{\nu}{r_1}$} K_{\nu}(k r_1)
  + \alpha K_{\nu}(k r_1)
  \,, \\
  m_{12} &:= -k I_{\nu}'(k r_1) + \mbox{$\frac{\nu}{r_1}$} I_{\nu}(k r_1)
  + \alpha I_{\nu}(k r_1)
  \,, \\
  m_{21} &:= k K_{\nu}'(k r_2) - \mbox{$\frac{\nu}{r_2}$} K_{\nu}(k r_2) 
  + \alpha K_{\nu}(k r_2) 
  \,, \\
  m_{22} &:= k I_{\nu}'(k r_2) - \mbox{$\frac{\nu}{r_2}$} I_{\nu}(k r_2)
  + \alpha I_{\nu}(k r_2) 
  \,.
\end{align*}
Non-trivial solutions of~\eqref{homo} 
correspond to the values of~$k$ for which the determinant
of the square matrix vanishes, \ie
\begin{equation}\label{lines}
  F(k,\alpha)  := m_{11} m_{22} - m_{21} m_{12} = 0
  \,,
\end{equation}
yielding an implicit equation for~$k$ as a function of~$\alpha$.

As a matter of fact, any eigenvalue corresponding to a rotationally symmetric 
eigenfunction in the Robin shell is determined as a solution 
$k=k(\alpha)$ of~\eqref{lines}.
We are interested in the largest of these solutions,
which determines $\lambda_1^\alpha(A_{r_1,r_2})$. 

\subsection*{Expanding Bessel functions} 
By \cite[Sec.~9.6.29]{Abramowitz-Stegun},
the derivative of Bessel functions is again 
a combination of Bessel functions of different orders, 
\begin{equation}\label{Bessel.derivative}
  I_{\nu}' = \frac{1}{2} \, (I_{\nu-1}+I_{\nu+1})
  \,, \qquad
  K_{\nu}' = -\frac{1}{2} \, (K_{\nu-1}+K_{\nu+1})
  \,.
\end{equation}
By \cite[Sec.~9.7.1]{Abramowitz-Stegun},
for Bessel functions of arbitrary order~$\mu$,
we have the factorisation
\begin{equation}\label{as0}
  I_\mu(z) = \frac{1}{\sqrt{2\pi z}} \, e^{z} \, \tilde{I}_\mu(z)
  \,, \qquad
  K_\mu(z) = \sqrt{\frac{\pi}{2 z}} \, e^{-z} \, \tilde{K}_\mu(z)
  \,,
\end{equation}
where $\tilde{I}_\mu$ and $\tilde{K}_\mu$ admit the asymptotic expansions
\begin{equation}\label{as}
\begin{aligned}
  \tilde{I}_\mu(z) &=
  1 - \frac{4\mu^2-1}{8z} 
  + \frac{(4\mu^2-1)(4\mu^2-9)}{2(8z)^2}
  + \mathcal{O}(z^{-3})
  \,,
  \\
  \tilde{K}_\mu(z) &= 
  1 + \frac{4\mu^2-1}{8z} 
  + \frac{(4\mu^2-1)(4\mu^2-9)}{2(8z)^2}
  + \mathcal{O}(z^{-3})
  \,,
\end{aligned}
\end{equation}
as $z \to +\infty$. 
Using~\eqref{Bessel.derivative} and~\eqref{as0}, 
we can rewrite the implicit equation~\eqref{lines} as
\begin{equation}\label{lines.new}
  \tilde{m}_{11} \tilde{m}_{22}
  - e^{-2k(r_2-r_1)} \, \tilde{m}_{21} \tilde{m}_{12}
  = 0
  \,,
\end{equation}
where
\begin{align*}
  \tilde{m}_{11} &:= \mbox{$\frac{1}{2}$} \, k  
  \big[\tilde{K}_{\nu-1}(k r_1)+\tilde{K}_{\nu+1}(k r_1)\big]
  + \mbox{$\frac{\nu}{r_1}$} \tilde{K}_{\nu}(k r_1)
  + \alpha \tilde{K}_{\nu}(k r_1)
  \,, \\
  \tilde{m}_{12} &:= -\mbox{$\frac{1}{2}$} \, k 
  \big[\tilde{I}_{\nu-1}(k r_1)+\tilde{I}_{\nu+1}(k r_1)\big]
  + \mbox{$\frac{\nu}{r_1}$} \tilde{I}_{\nu}(k r_1)
  + \alpha \tilde{I}_{\nu}(k r_1)
  \,, \\
  \tilde{m}_{21} &:= - \mbox{$\frac{1}{2}$} \, k 
  \big[\tilde{K}_{\nu-1}(k r_2)+\tilde{K}_{\nu+1}(k r_2)\big]
  - \mbox{$\frac{\nu}{r_2}$} \tilde{K}_{\nu}(k r_2) 
  + \alpha \tilde{K}_{\nu}(k r_2) 
  \,, \\
  \tilde{m}_{22} &:= \mbox{$\frac{1}{2}$} \, k 
  \big[\tilde{I}_{\nu-1}(k r_2)+\tilde{I}_{\nu+1}(k r_2)\big]
  - \mbox{$\frac{\nu}{r_2}$} \tilde{I}_{\nu}(k r_2)
  + \alpha \tilde{I}_{\nu}(k r_2) 
  \,.
\end{align*}
It follows from~\eqref{step0} and~\eqref{as}
that the second term on the left hand side of~\eqref{lines.new} 
vanishes in the limit $\alpha \to -\infty$.
Hence, $k=k(\alpha)$ satisfies 
\begin{equation}\label{limit}
  \lim_{\alpha \to -\infty} f(k,\alpha) = 0
\end{equation}
with 
$
  f(k,\alpha) 
  := \tilde{m}_{11} \tilde{m}_{22}
$.
Using~\eqref{as}, it is tedious but straightforward to  
verify the expansion
\begin{equation}\label{asf}
\begin{aligned}
  f(k,\alpha) = \ & 
  k^2 \Bigg[
  1 + \frac{d^2-4d+7}{8k} \left(\frac{1}{r_1}-\frac{1}{r_2}\right)
  - \frac{(d^2-4d+7)^2}{64 k^2 r_1 r_2}
  \\
  & \qquad + \frac{d^4-8d^3+38d^2-88d+57}{128 k^2}
  \left(\frac{1}{r_1^2}+\frac{1}{r_2^2}\right)
  +\mathcal{O}(k^{-3})
  \Bigg]
  \\
  & + 2k\alpha \Bigg[
  1 + \frac{d^2-4d+5}{8k} \left(\frac{1}{r_1}-\frac{1}{r_2}\right)
  \\
  & \qquad + \frac{(d^2-4d+7)(d^2-4d+3)}{128 k^2}
  \left(\frac{1}{r_1}-\frac{1}{r_2}\right)^2
  +\mathcal{O}(k^{-3})
  \Bigg]
  \\
  & + \alpha^2 \Bigg[
  1 + \frac{d^2-4d+3}{8k} \left(\frac{1}{r_1}-\frac{1}{r_2}\right)
  - \frac{(d^2-4d+3)^2}{64 k^2 r_1 r_2}
  \\
  & \qquad + \frac{d^4-8d^3+14d^2+8d-15}{128 k^2}
  \left(\frac{1}{r_1^2}+\frac{1}{r_2^2}\right)
  +\mathcal{O}(k^{-3})
  \Bigg]
  \\ 
  & + \nu k \Bigg[
  \frac{1}{r_1}-\frac{1}{r_2} - \frac{d^2-4d+7}{4 k r_1 r_2}
  \\
  & \qquad + \frac{d^2-4d+3}{8k} \left(\frac{1}{r_1^2}+\frac{1}{r_2^2}\right)
  +\mathcal{O}(k^{-2})
  \Bigg]
  \\ 
  & - \frac{\nu^2}{r_1 r_2} \left[1+\mathcal{O}(k^{-1})\right]
  \\
  & + \nu \alpha \left[
  \frac{1}{r_1}-\frac{1}{r_2} 
  + \frac{d^2-4d+3}{8 k} \left(\frac{1}{r_1}-\frac{1}{r_2}\right)^2
  +\mathcal{O}(k^{-2})
  \right]
  .
\end{aligned}
\end{equation}
Now we proceed in five steps.

\subsection*{Step 1} 
From~\eqref{limit}, we know that
$$
  \lim_{\alpha \to -\infty} \frac{f(k,\alpha)}{k^4} = 0
  \,.
$$ 
On the other hand, the asymptotics~\eqref{asf} 
together with~\eqref{step0} gives
$$
  \frac{f(k,\alpha)}{k^4} 
  = \left(\frac{\alpha}{k^2}\right)^2 + \mathcal{O}(k^{-1})
  \,.
$$
Hence, we conclude with an improvement upon~\eqref{step0},
\begin{equation}\label{step1}
  \frac{\alpha}{k^2} = o(1)
  \,.
\end{equation}

\subsection*{Step 2} 
We continue with dividing~$f$ by~$k^3$
and use the previous result~\eqref{step1}
to conclude from~\eqref{asf}
$$
  \frac{f(k,\alpha)}{k^3} 
  = \frac{\alpha^2}{k^3} + o(1)
  \,.
$$
Recalling~\eqref{limit}, we thus get
\begin{equation}\label{step2}
  \frac{\alpha^2}{k^3} = o(1)
  \,.
\end{equation}

\subsection*{Step 3} 
In the same vein, we conclude from
$$
  \frac{f(k,\alpha)}{k^2} 
  = \left(1+\frac{\alpha}{k}\right)^2 + o(1)
$$
and~\eqref{limit} the asymptotics
\begin{equation}\label{step3}
  1+\frac{\alpha}{k} = o(1)
  \,.
\end{equation}

\subsection*{Step 4} 
Using the previous results, we get from~\eqref{asf}
$$
  \frac{f(k,\alpha)}{k} 
  = \frac{(k+\alpha)^2}{k} + o(1)
  \,,
$$
so that~\eqref{limit} yields
\begin{equation}\label{step4}
  \frac{(k+\alpha)^2}{k} = o(1)
  \,.
\end{equation}

\subsection*{Step 5} 
Finally, using the previously established results,
a tedious but straightforward computation enables us to 
derive the following asymptotics
$$
\begin{aligned}
  f(k,\alpha)
  &= (k+\alpha)^2 + \frac{d-1}{2} (k+\alpha) \left(\frac{1}{r_1}-\frac{1}{r_2}\right)
  - \frac{(d-1)^2}{4 r_1 r_2} 
  + o(1)
  \\
  &= \left(k+\alpha-\frac{d-1}{2 r_2} \right) 
  \left(k+\alpha+\frac{d-1}{2 r_1} \right)
  + o(1)
  \,.
\end{aligned}
$$
From~\eqref{limit}, we thus eventually obtain
\begin{equation}\label{step5}
  k+\alpha-\frac{d-1}{2r_2} = o(1)
  \qquad \mbox{or} \qquad
  k+\alpha+\frac{d-1}{2r_1} = o(1)
  \,.
\end{equation}
Here the first asymptotics corresponds to 
the lowest eigenvalue of the Robin Laplacian
(because~$k$ will be larger in this limit),
while the second corresponds to another 
rotationally symmetric eigenfunction.

\medskip
Summing up, we have 
\begin{Theorem}\label{Thm.as}
Given positive numbers $r_1 < r_2$ and $r$,
we have the eigenvalue asymptotics
\begin{align}
  \lambda_1^\alpha(A_{r_1,r_2}) 
  &= -\alpha^2 + \frac{d-1}{r_2} \, \alpha + o(\alpha)
  && \mbox{(spherical shell)} 
  \,,
  \label{as.anulus} \\
  \lambda_1^\alpha(B_r) 
  &=
  -\alpha^2 + \frac{d-1}{r} \, \alpha + o(\alpha)
  && \mbox{(ball)} 
  \,.
  \label{as.ball}
\end{align}
as $\alpha \to -\infty$.
\end{Theorem}
%
Indeed, \eqref{as.anulus}~follows from~\eqref{step5}, while
\eqref{as.ball}~can be \emph{formally} deduced from \eqref{as.anulus}
by letting~$r_1$ go to zero.
To check it rigorously, one can proceed in the same (in fact easier) way 
as we did here for the spherical shells,
using now the implicit equation
$$
  k I_\nu'(kr) 
  - \mbox{$\frac{\nu}{r}$} I_\nu(kr) 
  + \alpha I_\nu(kr) = 0 
$$
instead of~\eqref{lines},
with $k:=\sqrt{-\lambda_1^\alpha(B_r)}$.
We omit the proof, which is a mere variation of the previous one.

\subsection*{Proof of Theorem~\ref{mainthm1} 
(disproval of Conjecture~\ref{Conj})}
With the asymptotics given by Theorem~\ref{Thm.as}, the disproval
of Conjecture~\ref{Conj} now becomes quite straightforward.
Consider a spherical shell $A_{r_1,r_2}$ with $0<r_1<r_2$
and a ball~$B_r$ with the same volume,
\ie\ $|A_{r_1,r_2}|=|B_r|$.    
Then necessarily $r_2 > r$
and Theorem~\ref{Thm.as} implies that 
the opposite inequality
$\lambda_1^\alpha(B_r) < \lambda_1^\alpha(A_{r_1,r_2})$
holds for all sufficiently large negative~$\alpha$.
\hfill\qed

\section{A universal upper bound in two dimensions}\label{Sec.bound}
%
To prove Conjecture~\ref{Conj} for 
planar domains and
a certain range of the parameter~$\alpha$,
we first adapt the original idea of 
Payne and Weinberger~\cite{Payne-Weinberger_1961}
in the Dirichlet case 
(which formally corresponds to $\alpha = +\infty$
in the present setting)
to use a test function in~\eqref{Rayleigh} whose level lines 
are parallel to a component of~$\partial\Omega$.
In some aspects
we rather follow the modern approach to parallel coordinates
developed by Savo~\cite{Savo_2001}.

Throughout this section, we assume that~$\Omega$ is of class~$C^2$,
so that the curvature of~$\partial\Omega$ is everywhere well defined.
For such domains the boundary~$\partial\Omega$ will, in general, be
composed of a finite union of $C^2$-smooth Jordan curves $\Gamma_0, \Gamma_1, \dots, \Gamma_N$,
$N \geq 0$, where~$\Gamma_0$ is the \emph{outer boundary}
\ie~$\Omega$ lies in the \emph{interior}~$\Omega_0$ of~$\Gamma_0$.
If $N=0$, then~$\Omega$ is simply connected and $\Omega=\Omega_0$. 
We denote by $L_0:=\mathcal{H}^1(\Gamma_0)$ the \emph{outer perimeter,
where~$\mathcal{H}^1(\cdot)$ stands for the $1$-dimensional Hausdorff measure.
By the isoperimetric inequality, we have
\begin{equation}\label{isoperimetric}  
  L_0^2 \geq 4 \pi A_0
  \,,
\end{equation}
where $A_0:=|\Omega|$ denotes the area of~$\Omega$
($A_0$~does not denote the area of~$\Omega_0$!).
Here and in the sequel~$|\cdot|$ stands for the $2$-dimensional Lebesgue measure.}

Following~\cite{Payne-Weinberger_1961},
we introduce parallel coordinates based at the outer boundary~$\Gamma_0$ only.
This approach allows us to parametrise domains 
with the same fixed area by their outer perimeter alone, in
the sense that we will be able to relate their first Robin eigenvalue 
to the first Neumann-Robin eigenvalue
of an annulus whose radii depend on the outer perimeter of the original domain. 
We are then able to control the behaviour of this
second eigenvalue problem in a uniform way with respect to the perimeter, by comparing it to the problem
on the disk.

\subsection*{From boundaries to cut-loci}
Define the map $\Phi:\Gamma_0\times[0,\infty) \to \Real^2$
by setting
\begin{equation}
  \Phi(s,t) := s - n(s) \, t 
  \,,
\end{equation}
where, as before, $n(s)$ is the unit vector,
normal to~$\partial\Omega$ at~$s$ 
and oriented outside~$\Omega$. 
Define the \emph{cut-radius} map $c:\Gamma_0 \to (0,\infty)$
by the property that the segment $t \mapsto \Phi(s,t)$
minimises the distance from~$\Gamma_0$
if and only if $t \in [0,c(s)]$.	
The cut-radius map is known to be continuous
and we clearly have $\max c = R$,
where~$R$ is the \emph{inner radius} of~$\Omega_0$
(\ie~the radius of the largest inscribed disk). 
The \emph{cut-locus} 
$
  \Cut(\Gamma_0) := \{\Phi(s,c(s)) : s \in \Gamma_0\}
$
is a closed subset of~$\Omega_0$ of measure zero.
The map~$\Phi$, when restricted to the open set 
$
  U := \{ (s,t) \in \Gamma_0 \times (0,\infty)  
  : 0 < t < c(s)
  \} 
$
is a diffeomorphism onto $\Phi(U) = \Omega_0 \setminus \Cut(\Gamma_0)$.
The pair~$(s,t)$ are called the \emph{parallel coordinates}
based at~$\Gamma_0$.

The Jacobian of~$\Phi$ is given by 
\begin{equation}\label{Jacobian}
  h(s,t) := 1-\kappa(s) \, t
  \,,
\end{equation}
where~$\kappa$ is the curvature of~$\Gamma_0$
as defined by the Frenet equation $\tau'(s) = - \kappa(s) \;\! n(s)$,
with~$\tau$ being the unit tangent vector field along~$\Gamma_0$ 
such that the pair $(\tau,-n)$ is positively oriented.
Note that the convention is chosen is such a way that $\kappa \geq 0$
if~$\Omega_0$ is convex. 
We have the uniform bound 
\begin{equation}\label{Jacobian.bound}
  \|h\|_{L^\infty(U)} \leq 1 + \|\kappa\|_{L^\infty(\Gamma_0)} \, R
  \,.
\end{equation}

Let $\rho:\Omega_0\to(0,\infty)$ be the \emph{distance function}
from the outer boundary~$\Gamma_0$, 
\ie~$\rho(x) = \dist(x,\Gamma_0)$.
Let
$
  A(t) := |\{x\in\Omega : 0 < \rho(x) < t \}|
$
denote the area of the portion of a $t$-neighbourhood of~$\Gamma_0$ lying in~$\Omega$.
Clearly, $A(R)$~equals the total area $A_0:=|\Omega|$.
We also introduce the length of the boundary curve $\{\rho(x)=t\}$
that lies inside~$\Omega$:  
$$
  L(t) := \mathcal{H}^1\big(\rho^{-1}(t) \cap \Omega \big)
  = \int_{\{s \in \Gamma_0, \, t<c(s), \,\Phi(s,t) \in \Omega \}}
  h(s,t) \, ds
  \,.
$$
Clearly, $L(0)=L_0$ 
and, using~\eqref{Jacobian.bound}, 
we get the crude bound
$$
  L(t) \leq L_0 
  \left( 1 + \|\kappa\|_{L^\infty(\Gamma_0)} \, R \right)
  \,.
$$
Then, writing 
$
  |A(t_2) - A(t_1)| 
  = |
  \int_{t_1}^{t_2} L(t) \, dt
  |
$,
we see that~$A(t)$ is uniformly Lipschitz on~$[0,R]$
and, for almost every~$t$, 
\begin{equation}\label{coarea}
  A'(t)=L(t)
  \,.
\end{equation}

Now we pick a smooth function $\phi:[0,R] \to \Real$
and consider in~\eqref{Rayleigh}
the test-function $u = \phi \circ A \circ \rho$
which is Lipschitz in~$\Omega$.
Employing the parallel coordinates
together with the co-area formula 
(\cf~\cite[Eq.~(30)]{Savo_2001} for some more details),
we have
\begin{align*}
  \|u\|_{\sii(\Omega)}^2 
  &= 
  \int_0^R \phi(A(t))^2 \, A'(t) \, dt 
  \,, \qquad
  \\
  \|\nabla u\|_{\sii(\Omega)}^2 
  &= 
  \int_0^R \phi'(A(t))^2 \, A'(t)^3 \, dt 
  \,, \qquad
  \\
  \|u\|_{\sii(\partial\Omega)}^2 
  &\geq
  L_0 \, \phi(0)^2
  \,.
\end{align*}
Â«Here the inequality follows by neglecting 
the inner components $\Gamma_1, \dots, \Gamma_N$ of 
the boundary~$\partial\Omega$.

\subsection*{From domains to annuli}
The remarkable idea of~\cite{Payne-Weinberger_1961}
is to consider the change of variables
\begin{equation}\label{PW-idea}
  r(t) := \frac{\sqrt{L_0^2 - 4\pi A(t)}}{2\pi}
\end{equation}
on $[r_1,r_2]$ with
\begin{equation}\label{radii}
  r_1 := r(R) = \frac{\sqrt{L_0^2-4\pi A_0}}{2\pi}
  \,, \qquad
  r_2 := r(0) = \frac{L_0}{2\pi} 	
  \,.
\end{equation}
Note that transformation~\eqref{PW-idea} is well defined
for every $t \in [0,R]$ due to 
the isoperimetric inequality~\eqref{isoperimetric}.
Then, introducing $\psi(r) := \phi(L_0^2/(4\pi)-\pi r^2)$,
we have
\begin{equation}\label{general}
\begin{aligned}
  \|u\|_{\sii(\Omega)}^2 
  &= 
  2\pi\int_{r_1}^{r_2} \psi(r)^2 \, r \, dr
  \,, \qquad
  \\
  \|\nabla u\|_{\sii(\Omega)}^2 
  &= 
  2\pi\int_{r_1}^{r_2} \psi'(r)^2 \, r'(t)^2 \, r \, dr 
  \,, \qquad
  \\
  \|u\|_{\sii(\partial\Omega)}^2 
  &\geq
  L_0 \, \psi(r_2)^2
  \,.
\end{aligned}
\end{equation}

Following~\cite{Payne-Weinberger_1961},
we would like to estimate the extra factor~$r'(t)^2$ by~$1$
in order to be able to compare the original eigenvalue problem
with that on the annulus~$A_{r_1,r_2}$. 
Note that the radii in~\eqref{radii} are such that 
the annulus~$A_{r_1,r_2}$ has the same
area~$A_0$ as the original domain~$\Omega$,  
\ie\ $|A_{r_1,r_2}|=A_0 \equiv |\Omega|$.
\begin{Proposition}\label{Prop.crucial}
If~$\Omega$ is a planar domain of class~$C^2$, 
then $|r'(t)| \leq 1$
for almost every $t \in [0,R]$. 
\end{Proposition}
\begin{proof}
From~\eqref{coarea} we get,
for almost every $t \in [0,R]$, 
\begin{equation}\label{crucial}
  r'(t) = - \frac{L(t)}{\sqrt{L_0^2 - 4\pi A(t)}}
  \,.
\end{equation}
Since~$\Gamma_0$ is a simple Jordan curve, we have
$$
  \int_{\Gamma_0} \kappa(s) \, ds = 2 \pi
  \,.
$$ 
It then follows from~\eqref{Jacobian} that,
for almost every $t \in [0,R]$,
$$
  L(t) \leq L_0 - 2\pi t
  \qquad \mbox{and} \qquad
  A(t) = \int_0^t L(u) \, du
  \leq L_0 t-\pi t^2
  \,.
$$
Using the latter inequality in the former,
we conclude with the isoperimetric-type estimate
$$
  L(t)^2 \leq L_0^2 - 4\pi A(t)
  \,.
$$
Putting this inequality into~\eqref{crucial},
we obtain the desired result.
\end{proof}
\begin{Remark}
Except in the case of simply-connected~$\Omega$, it will not be possible
to prove Proposition~\ref{Prop.crucial} if we build the parallel coordinates
based at the entire boundary~$\partial\Omega$ instead of~$\Gamma_0$.
Indeed, $|r'(t)|$ may become larger than~$1$ 
whenever~$\Omega$ is not simply connected
and~$L_0$ is replaced by the total perimeter~$|\partial\Omega|$
(in fact, $L(t)=|\partial\Omega|$ for~$\Omega$ itself being any annulus).
\end{Remark}

Employing Proposition~\ref{Prop.crucial} in~\eqref{general}
and assuming $\alpha \leq 0$,
we obtain from~\eqref{Rayleigh} the upper bound
\begin{equation}\label{Rayleigh.annulus}
  \lambda_1^\alpha(\Omega)
  \leq \inf_{\psi\not=0} \frac{\displaystyle
  \int_{r_1}^{r_2} \psi'(r)^2 \, r \, dr 
  +\alpha \, r_2 \, \psi(r_2)^2}
  {\displaystyle
  \int_{r_1}^{r_2} \psi(r)^2 \, r \, dr}
  =: \mu_1^\alpha(A_{r_1,r_2})
  \,,
\end{equation}
where the infimum is taken over all 
smooth function~$\psi$ on $[r_1,r_2]$.
The notation $\mu_1^\alpha(A_{r_1,r_2})$ refers to the fact
that the infimum is attained for the first eigenfunction~$\psi_1^\alpha$
of the Laplacian in the annulus~$A_{r_1,r_2}$,
subject to the Robin boundary condition with~$\alpha$ 
on the outer boundary $\partial B_{r_2}$
and the Neumann boundary condition on the inner boundary $\partial B_{r_1}$.
Again, we choose~$\psi_1^\alpha$ 
normalised to~$1$ in $\sii((r_1,r_2),r\;\!dr)$.

We have thus proven
\begin{Theorem}\label{Thm.better}
Let $\alpha \leq 0$.
For any planar domain~$\Omega$ of class~$C^2$,
$$
  \lambda_1^\alpha(\Omega) \leq \mu_1^\alpha(A_{r_1,r_2})
  \,,
$$
where~$A_{r_1,r_2}$ is the annulus of the same area as~$\Omega$
with radii~\eqref{radii}.
\end{Theorem}

\subsection*{From annuli to disks}
To obtain Conjecture~\ref{Conj} from Theorem~\ref{Thm.better},
we would need the inequality 
\begin{equation}\label{better}
  \mu_1^\alpha(A_{r_1,r_2}) \leq \lambda_1^\alpha(B_{r_3})
\end{equation}
to hold, where~$B_{r_3}$ is the disk of the same area 
as the annulus~$A_{r_1,r_2}$
(which has the same area~$A_0$ as the original domain~$\Omega$),
\ie\ 
\begin{equation}\label{r3A0}
  r_3:=\sqrt{\frac{A_0}{\pi}}
\end{equation}
However, \eqref{better}~is false in general!
To see this, one can proceed as in Section~\ref{Sec.as}
and establish the asymptotics
\begin{align}
  \lambda_1^\alpha(B_{r_3}) 
  &= -\alpha^2 + \frac{\alpha}{r_3} + o(\alpha)
  \qquad\qquad \mbox{(Robin disk)} 
  \,,
  \label{R-disk} \\
  \mu_1^\alpha(A_{r_1,r_2}) 
  &= -\alpha^2 + \frac{\alpha}{r_2} + o(\alpha)
  \qquad\qquad \mbox{(Neumann-Robin annulus)} 
  \,,
  \label{RN-annulus}
\end{align}
as $\alpha \to -\infty$.
Unless $r_1=0$ (so that the annulus coincides with the disk),
we always have $r_3 < r_2$, so that \eqref{R-disk} and \eqref{RN-annulus}
imply that actually an inequality opposite to~\eqref{better}
holds for all sufficiently large (negative) $\alpha$.
Hence, using the same argument as in Section~\ref{Sec.as}
to disprove Conjecture~\ref{Conj},
we see that the method of parallel coordinates 
is intrinsically not adequate to prove the conjecture in the whole range of~$\alpha$
(even for simply connected~$\Omega$).

On the other hand, \eqref{better}~\emph{does} hold
provided that $\alpha<0$ is small enough,
so that Theorem~\ref{Thm.better} implies Conjecture~\ref{Conj} 
for each family of domains with the same area.		
This follows from the asymptotics
\begin{align}
  \lambda_1^\alpha(B_{r_3}) 
  &= 2\,\alpha \, \frac{r_3}{r_3^2} + \mathcal{O}(\alpha^2)
  \qquad\qquad \mbox{(Robin disk)} 
  \,,
  \label{R-disk.small} \\
  \mu_1^\alpha(A_{r_1,r_2}) 
  &= 2\,\alpha \, \frac{r_2}{r_3^2} + \mathcal{O}(\alpha^2)
  \qquad\qquad \mbox{(Neumann-Robin annulus)} 
  \,,
  \label{RN-annulus.small}
\end{align}
as $\alpha \to 0$,
which can be easily established by analytic perturbation theory, 
plus the fact that
$r_3<r_2$ unless the annulus coincides with the disk.

Summing up, we have 
\begin{Theorem}\label{Thm.small}
For any planar domain~$\Omega$ of class~$C^2$,
there exists a negative number $\alpha_0$, depending on $A_{0}$ and $L_{0}$,
such that 
$$
  \lambda_1^\alpha(\Omega) \leq \lambda_1^\alpha(B)
$$
holds for all $\alpha \in [\alpha_{0},0]$,
where~$B$ is the disk of the same area as~$\Omega$.
\end{Theorem}
\begin{Remark}
For general domains in $\Real^{d}$, 
and based on the expression for the derivative of $\lambda_{1}^{\alpha}$ with respect
to $\alpha$ at zero, namely,
\[
  \left.
  \frac{\ds d}{\ds d\alpha}\lambda_{1}^{\alpha}(\Omega)
  \right|_{\alpha=0} 
  = \frac{\ds |\partial \Omega|}{\ds |\Omega|}
  \,,
\]
it is possible to obtain a similar result 
but where~$\alpha_{0}$ is replaced by a value~$\alpha_{\Omega}$ 
depending on each domain $\Omega$.
\end{Remark}

\subsection*{Uniform behaviour for small \texorpdfstring{$\alpha$}{alpha}}

In order to be able to prove Theorem~\ref{mainthm2}, 
it still remains to show that the constant~$\alpha_0$ from Theorem~\ref{Thm.small}
can be made independent of~$L_0$.
More specifically, the neighbourhood of zero in which~\eqref{better} 
holds for negative $\alpha$ does not degenerate either
when the annulus approaches the disk 
or when it becomes unbounded. 
If the value of~$\alpha_0$ remains bounded away from zero uniformly
in these two instances, this together with Theorem~\ref{Thm.small} yields the desired result.

\begin{proof}[Proof of Theorem~\ref{mainthm2}]
With~\eqref{radii} and~\eqref{r3A0} in mind, we write
\[
  r_{1} = \sqrt{2\eps r_{3}+\eps^2} 
  \qquad\mbox{and}\qquad 
  r_2 = r_{3}+\eps,
\]
where $\eps$ is a positive parameter. 
Note that since we are interested in the case of fixed area, 
$r_{3}$~is fixed and equal to $\sqrt{A_{0}/\pi}$. 
In this notation, 
and in an analogous way as in the derivation of~\eqref{lines},
the equation yielding the first eigenvalue 
$\mu_{1}^{\alpha}(A_{r_1,r_2}) =:-k^2$ 
of $A_{r_1,r_2}$
with Neumann-Robin boundary conditions becomes
\begin{multline}\label{murobneu}
  K_{1}\big( k \sqrt{2\eps r_{3}+\eps^{2}}\big)
  \big\{ k I_{1}\left[k(r_{3}+\eps)\right]
  +\alpha I_{0}\left[k(r_{3}+\eps)\right]\big\} 
  \\
  - I_{1}\big( k \sqrt{2\eps r_{3}+\eps^{2}}\big)
  \big\{ k K_{1}\left[k(r_{3}+\eps)\right]
  -\alpha K_{0}\left[k(r_{3}+\eps)\right]\big\}=0
  \,.
\end{multline}
Our aim is to show that there exists a negative number $\alpha_{*}$
such that the curve 
$\alpha \mapsto \mu_1^\alpha(A_{r_1,r_2})$
stays below the curve $\alpha \mapsto \lambda_1^\alpha(B_{r_3})$
corresponding to the eigenvalue of the Robin disk for all $\alpha$ in 
$(	\alpha_{*},0)$ and all $\eps>0$. 
In what follows, 
we shall denote these curves for a given annulus and the disk by $\Gamma_{A}$ and $\Gamma_{B}$, respectively.

Since both $\Gamma_{A}$ and $\Gamma_{B}$ are analytic and, 
for any given positive $\eps$, 
we have that the derivative of the Neumann-Robin
problem in the annulus
with respect to $\alpha$ at zero is smaller than that for the Robin disk, 
\cf~\eqref{R-disk.small} and~\eqref{RN-annulus.small},
it follows that we only need to consider the limiting
cases where $\eps$ approaches zero and infinity. 
Furthermore, in order for a switch to occur, that is, 
for one of the curves $\Gamma_{A}$ to go above $\Gamma_{B}$, 
there has to be an intersection between the two curves. 
We shall now study the behaviour of the largest 
(\ie\ closest to zero) of these intersection points.

We first consider the situation as $\eps$ goes to infinity. 
From the variational formulation~\eqref{Rayleigh.annulus}, for instance, 
we have that the curves $\Gamma_{A}$ are concave in~$\alpha$. 
On the other hand, their derivatives with respect to $\alpha$ at zero 
are increasing with~$\eps$. 
Thus, if we pick the curve corresponding to a specific annulus 
and consider its tangent at the origin, 
we have that this tangent must also intersect~$\Gamma_{B}$ at one 
(and only one) point to the left of zero, say~$\alpha_{1}$. 
Due to concavity, it now follows that any curve~$\Gamma_{A}$ 
for a larger value of~$\eps$ intersecting~$\Gamma_{B}$ 
must do so to the left of~$\alpha_{1}$.

The situation as $\eps$ becomes small is more complex, 
requiring a careful analysis of the behaviour of the intersection point 
as $\eps$ becomes small. 
To this end, we need to compare the solutions of~\eqref{murobneu} 
with that of the equation for the disk given
by
\[
 k \;\! I_{1}(k r_{3})+ \alpha \;\! I_{0}(k r_{3})=0 \,.
\]
Solving in both equations with respect to~$\alpha$ 
and equating the results yields the following equation in~$k$
\begin{equation*}
  \frac{I_{1}\left[k(r_{3}+\eps)\right]
  K_{1}\big( k \sqrt{2\eps r_{3}+\eps^{2}}\big)
  -I_{1}\left( k \sqrt{2\eps r_{3}+\eps^{2}}\right)
  K_{1}\left[k(r_{3}+\eps)\right]}
  {I_{1}\big( k \sqrt{2\eps r_{3}+\eps^{2}}\big)
  K_{0}\left[k(r_{3}+\eps)\right]
  + I_{0}\left[k(r_{3}+\eps)\right]
  K_{1}\big( k \sqrt{2\eps r_{3}+\eps^{2}}\big)}  
  = \frac{\ds I_{1}(k r_{3})}{\ds I_{0}(k r_{3})}.
\end{equation*}

Solutions of the intersection problem we are concerned with 
must thus satisfy an equation of the form $F(\eps,k,r_{3})=0$ with
the function $F$ defined as
\begin{align*}
  \lefteqn{F(\eps,k,r_{3}) :=} 
  \\
  & I_{0}(k r_{3}) 
  \left\{  
  I_{1}\big( k \sqrt{2\eps r_{3}+\eps^{2}}\big)
  K_{1}\left[k(r_{3}+\eps)\right] 
  -I_{1}\left[k(r_{3}+\eps)\right]
  K_{1}\big( k \sqrt{2\eps r_{3}+\eps^{2}}\big)
  \right\} 
  \\ 
  & + I_{1}(k r_{3})
  \left\{ 
  I_{1}\big( k \sqrt{2\eps r_{3}+\eps^{2}}\big)
  K_{0}\left[k(r_{3}+\eps)\right] 
  +I_{0}\left[k(r_{3}+\eps)\right]
  K_{1}\big( k \sqrt{2\eps r_{3}+\eps^{2}}\big)
  \right\}
\end{align*}
and the solution in $\alpha$ being then given by
\[
 \alpha = -k \, \frac{\ds I_{1}(k r_{3})}{\ds I_{0}(k r_{3})}.
\]
Note that since the function $y\mapsto -y I_{1}(y)/I_{0}(y)$ 
vanishes at zero and is strictly increasing for positive $y$, 
and $r_{3}$ is fixed,
showing that there are no intersections 
between $\Gamma_{A}$ and $\Gamma_{B}$ for $\alpha$ close to zero 
is equivalent to showing that there are
no solutions of the equation $F(\eps,k,r_{3})=0$ 
for small (positive)~$k$.

We shall now write $F$ as
\[
  F(\eps,k,r_{3}) = a(k,r_{3}) \sqrt{\eps} + b(\eps,k,r_{3})
  \,,
\]
where
\begin{align*}
  a(k,r_{3}) 
  & := \frac{1}{\sqrt{r_{3}}} \, a_{0}(k r_{3})
  \,,
  \\
  a_{0}(y) 
  & := \frac{-y I_{0}^{2}(y) + I_{0}(y)I_{1}(y) 
  + y\left[ 1 + I_{1}^{2}(y)\right]}{\sqrt{2} \;\! y}
  \,,
\end{align*}
and
\[
  b(\eps,k,r_{3}) 
  := F(\eps,k,r_{3})-a(k,r_{3})\sqrt{\eps} 
  = \mathcal{O}(\eps^{3/2}\log\eps)
  \quad\mbox{as}\quad \eps \to 0.
\]
We thus have that the equation $F(\eps,k,r_{3})=0$ 
is equivalent to solving
\begin{equation}\label{auxeq1}
  a_{0}(k r_{3}) + \frac{b(\eps,k,r_{3})}{\sqrt{\eps}}=0
  \,.
\end{equation}
It is not difficult to check that the second term above 
is such that there exists a positive constant $C$ for which
\[
 \left|\frac{\ds b(\eps,k,r_{3})}
 {\ds \sqrt{\eps}}\right|\leq C\eps\log\eps
\]
is satisfied for all $(\eps,k)$ 
in $S:=[0,\eps_{1}]\times[0,k_{1}]$, 
for some positive numbers~$\eps_{1}$ and~$k_{1}$ 
(and fixed $r_{3}$). 
In particular, note that all the functions $K_{i}$ $(i=0,1)$ 
which are those appearing in $b(\eps,k,r_{3})$ 
which are unbounded at zero, 
are always multiplying by a term containing $I_{1}$, 
and so the singularities in $K_{i}$ are cancelled out; 
the remaining part is a direct consequence 
from the asymptotic expansions in~$\eps$ around zero.
It thus follows that any positive solution 
$k=k(\eps)$ of~\eqref{auxeq1} 
(and thus also of $F(\eps,k,r_{3})=0$)
must be such that $a_{0}(k(\eps)r_{3})$ 
also converges to zero as~$\eps$ goes to zero. 
Since the function $a_{0}$ has only one
(strictly) positive zero, say $y_{0}$, 
we have that as $\eps$ approaches zero the intersection 
between the two curves must be such that 
$k r_{3}$ approaches~$y_{0}$, 
while~$\alpha$ converges to $-y_{0}I_{1}(y_{0})/(r_{3}I_{0}(y_0))$.

This, together with the behaviour as $\eps$ goes to infinity described above, proves Theorem~\ref{mainthm2}.
\end{proof}

\subsection*{Acknowledgment}
We are grateful to Vincenzo Ferone and Cristina Trombetti
for pointing out that our original proof of Theorem~\ref{mainthm2}
restricted to simply connected domains could be extended to the present general case.

%


\providecommand{\bysame}{\leavevmode\hbox to3em{\hrulefill}\thinspace}
\providecommand{\MR}{\relax\ifhmode\unskip\space\fi MR }
\providecommand{\MRhref}[2]{%
  \href{http://www.ams.org/mathscinet-getitem?mr=#1}{#2}
}
\providecommand{\href}[2]{#2}

\end{document}